\documentclass[11pt]{amsart}
\pdfoutput=1
\usepackage{amsthm,amsmath,amsxtra,amscd,amssymb,xypic,xparse,xspace,mathtools}
\usepackage[all]{xy}
\usepackage{tikz-cd}
\numberwithin{equation}{section}
\usepackage{bm}
\usepackage{amsfonts}
\usepackage{hyperref}
\usepackage{cleveref}
\usepackage{tabularx}
\usepackage{booktabs}
\usepackage{caption,subcaption}
\usepackage{enumitem}
\usepackage{float}

\theoremstyle{plain}
\newtheorem{thm}{Theorem}[section]
\newtheorem{lemma}[thm]{Lemma}
\newtheorem{prop}[thm]{Proposition}
\newtheorem{cor}[thm]{Corollary}

\newtheorem{conj}[thm]{Conjecture}
\newtheorem{remark}[thm]{Remark}

\newtheorem{convention}[thm]{Convention}

\theoremstyle{definition}
\newtheorem{definition}[thm]{Definition}

\newtheorem{exax}[thm]{Example}

\newtheoremstyle{colon}%
{}
{}
{\itshape}
{}
{\bfseries}
{:}
{ }
{}

\theoremstyle{colon}
\newtheorem*{claim}{Claim}

\def\={\;=\;}  \def\+{\,+\,}

\newcommand{\codim}{\operatorname{codim}}
\newcommand{\GL}{\operatorname{GL}}

\newcommand{\Res}{\operatorname{Res}}

\newcommand{\ord}{\operatorname{ord}}

\newcommand{\rank}{\operatorname{rank}}

\newcommand{\calH}{\mathcal H}

\newcommand{\calL}{\mathcal L}
\newcommand{\calM}{{\mathcal M}}

\newcommand{\calP}{\mathcal P}

\newcommand{\calR}{\mathcal R}

\newcommand{\calT}{{\mathcal T}}

\newcommand{\CC}{{\mathbb{C}}}

\newcommand{\PP}{{\mathbb{P}}}
\newcommand{\QQ}{{\mathbb{Q}}}
\newcommand{\RR}{{\mathbb{R}}}
\newcommand{\ZZ}{{\mathbb{Z}}}

\newcommand{\poles}{\underline{p}}
\newcommand{\zeroes}{\underline{z}}

\newcommand{\pr}{\operatorname{pr}}

\newcommand{\lG}{{\Gamma}}     

\newcommand{\Ram}{\operatorname{Ram}}

\newcommand{\hor}{{\mathrm {hor}}}
\newcommand{\ver}{{\mathrm {ver}}}

\DeclareDocumentCommand{\Ehor}{O{\lG}}{E^{\hor}(#1)}   
\DeclareDocumentCommand{\Ever}{O{\lG}}{E^{\ver}(#1)}  
\DeclareDocumentCommand{\Ehori}{O{i} O{\lG} }{E^{\hor}_{(#1)}(#2)}   

\DeclareDocumentCommand{\abEhor}{O{\lG}}{\widetilde{E^{\hor}}(#1)}

\newcommand\fib{\operatorname{fib}}

\DeclareDocumentCommand{\relhomZ}{O{X} O{\zeroes} O{\poles}}{H_1(#1\setminus #3,#2;\ZZ)}
\DeclareDocumentCommand{\relhomC}{O{X} O{\zeroes} O{\poles}}{H_1(#1\setminus #3,#2;\CC)}

\newcommand\St[1][(\mu)]{\calH #1}

\newcommand\logSt[1][(\mu)]{\calL #1}

\newcommand\PSt[1][(\mu)]{\PP\St[#1]}

\newcommand\Mgn[1][g,n]{\calM_{#1}}

\usepackage{color}

\newcommand\eqc{=_{\operatorname{comp}}}

\newcommand\Exc[1][(\mu)]{\operatorname{Exc}#1}

\begin{document}
\title{Bi-algebraic geometry of strata of differentials in genus zero}
\author{Frederik Benirschke}

\begin{abstract}
We study algebraic subvarieties of strata of differentials in genus zero satisfying algebraic relations among periods. The main results are  Ax-Schanuel and Andr{\'e}-Oort-type theorems in genus zero. As a consequence, one obtains several equivalent characterizations of bi-algebraic varieties. It follows that bi-algebraic varieties in genus zero are foliated by affine-linear varieties. Furthermore, bi-algebraic varieties with constant residues in strata with only simple poles are affine-linear.
Additionally, we produce infinitely many new linear varieties in strata of genus zero, including infinitely many new examples of meromorphic Teichm\"uller curves.

\end{abstract}
\maketitle

\tableofcontents

\section{Introduction}

Let $\mu$ be an integer partition of $2g-2$.
The {\em stratum of differentials} $\St$ is the moduli space parametrizing pairs $(X,\omega)$ where $X$ is a compact Riemann surface and $\omega$ a (possibly meromorphic) differential with 
\[
(\omega) = \sum_{i=1}^l \mu_ip_i, 
\]
where $\mu = (\mu_1,\ldots,\mu_l)$.
Associated with a differential $(X,\omega)$ is the relative cohomology class
\[
[\omega]\in H^1(X-P(\omega),Z(\omega);\ZZ).
\]
Here $Z(\omega)$ and $P(\omega)$ are the zeros and poles of $\omega$ in $X$, respectively.
It is known classically that, analytically locally, the stratum can be described in terms of relative cohomology.
More precisely,  choose a small open neighborhood $U\subseteq\St$ where the flat vector bundle with fiber $H_1(X-P(\omega),Z(\omega);\CC)$ is trivial.
After choosing a local frame $\gamma_1,\ldots,\gamma_d$ for the relative homology on $U$, one obtains a holomorphic map
\[
\phi_U:U \to \calP=\CC^d, (X,\omega)\mapsto \left(\int_{\gamma_1}\omega,\ldots,\int_{\gamma_d} \omega\right).
\]
For $U$ small enough, $\phi_U$ is a biholomorphism onto its image. One thus obtains local analytic coordinates on $\St$, called {\em period coordinates}.
The stratum is an algebraic variety, but period coordinates are only holomorphic coordinates.
An algebraic subvariety is called {\em bi-algebraic} if it agrees in local period coordinates with an algebraic subvariety $Z$ of $\calP$. The variety $Z$ is called an {\em algebraic model}.
This paper aims to understand the structure of bi-algebraic subvarieties of strata in genus zero.

Two algebraic maps associated with a stratum $\St$ in genus zero govern the behavior of bi-algebraic varieties: the {\em residue map }$R$ and the {\em twisted algebraic period map} $A$.

The {\em residue map} $R$ is a global algebraic map and sends a differential to the vector of residues at the marked poles.
Besides residues, a differential in genus zero also has relative periods $\int_{x_0}^{x_i} \omega$. The relative periods can be decomposed as
\[
\int_{x_0}^{x_i} \omega = F_i^{alg}+ F_i^{log},
\]
where $F^{alg}_i$ is a global algebraic function on $\St$, while $F_i^{log}$ depends on the choice of branch of logarithm and is only well-defined locally. To bypass the issue of locally choosing branches of logarithms,  we construct in \Cref{sec-periodsetup} a (locally closed) embedding of the stratum $\St$ into an algebraic torus $(\CC^*)^M$.
This allows to define the {\em twisted algebraic period map}\[
A:\St\times \CC^M \to \calP.
\]
The map $A$ is characterized by the property that it agrees with  local period coordinates if restricted to the graph of the exponential function \[
\Gamma\subseteq (\CC^*)^M\times \CC^M.\] The precise definition is in \Cref{eq:A}.

Our first result is an Ax-Schanuel theorem for strata in genus zero. 
Using the embedding $\St\subseteq (\CC^*)^M$, one defines a {\em  torus coset} $W\subseteq \St$ as the intersection of a translate of a subtorus $W'$ of $ (\CC^*)^M$ with $\St$.  Every torus coset has an {\em associated affine space } $V=\exp(W')\subseteq \CC^M$.
Similarly, a {\em residue coset} is a component of the intersection of a torus coset with a residue variety $R^{-1}(Z)$ for some algebraic variety $Z.$

Let \[\Omega =\{(s,A(s,v))\in \St\times \calP\,|\, (s,v)\in S\times V\}\] and $\Sigma\subseteq \Omega$ be the graph of $\pi\times \phi:\widetilde{S}\to \Omega$, where 
\begin{itemize}
    \item $V$ is the affine space associated with a minimal torus coset $W$ containing $S$,
    \item $\pi:\widetilde{S}\to S$ is the universal covering,
    \item $\phi:\widetilde{S}\to\calP$ is the period map obtained by integrating $\omega$ over a basis of relative homology.
    
\end{itemize}
Furthermore, let $
\pr_1:\Omega\to S$
be the projection.

\begin{thm}[Ax-Schanuel in genus zero]
\label{thm:ax}Let $S\subseteq \St$ be a residue coset in a stratum in genus zero.
Suppose $Z\subseteq \Omega$ is an algebraic variety. If an irreducible analytic component $U$ of $Z\cap \Sigma$ satisfies
\[
\codim_{\Omega} U < \codim_{\Omega} Z + \codim_{\Omega} \Sigma,
\]
then $\pr_1(U)$ is contained in a proper residue coset.
\end{thm}
Residue cosets are closely related to bi-algebraic varieties, but the two notions are generally different. 
In strata with only simple poles, every bi-algebraic variety is a residue coset, see \Cref{cor:rescoset}

Klingler and Lerer initiated the study of the bi-algebraic structure of strata in \cite{KlinglerLerer}, where the focus was on holomorphic strata. 
Bakker and Tsimerman establish an Ax-Lindenmann result for bi-algebraic varieties in any stratum, see  \cite[Thm. 5.4]{BakkerTsimerman}. In (loc. cit.), the result is stated for holomorphic differentials, but the proof works verbatim for meromorphic strata.

\subsection{The structure of bi-algebraic varieties}
The Ax-Schanuel theorem \cite{Ax} for the exponential function implies that bi-algebraic varieties for $\exp$ are exactly translates of rational subspaces.
A common feature of the different generalizations of  Ax-Schanuel theorems is that bi-algebraic varieties have many properties of linear varieties.
Because of residues, one cannot expect bi-algebraic varieties to be linear in general. Still, for bi-algebraic varieties in strata in genus zero, we will see that they are swept out by linear spaces in a strong sense. Hence, their algebraic models in period coordinates are closely related to ruled varieties. 
\begin{definition}
    A bi-algebraic variety $S\subseteq \St$ is {\em affine-linear}, if locally in period coordinates, it is defined by a finite union of affine-linear subspaces. Equivalently $S=\pi(Y)$ for some irreducible analytic component $Y$ of $\phi^{-1}(Z)$, where $Z\subseteq \calP$ is affine-linear.
\end{definition}

The following gives equivalent characterizations of bi-algebraic varieties in genus zero.
\begin{thm}[Characterization of bi-algebraic varieties]
\label{thm-characterization}Let $\St$
be a stratum in genus zero and $S\subseteq \St$  an irreducible algebraic variety.
Let $W\subseteq (\CC^*)^M$ be a minimal torus coset containing $S$ with associated affine space $V$.
The following are equivalent
\begin{enumerate}
    \item $S$ is bi-algebraic,
    \item $\dim A(S\times V) = \dim S$,
    \item For every $s\in S$ there exists an affine-linear  bi-algebraic subvariety $X_s\subseteq S$ containing $s$ with algebraic model $A(\{s\}\times V)\subseteq \calP$ in period coordinates.

\end{enumerate}
\end{thm}
Statement (3) has strong implications for the geometry of bi-algebraic subvarieties. In particular, there is an affine-linear bi-algebraic subvariety, potentially of dimension zero, through every point.

\begin{definition}
    An algebraic variety $S\subseteq \St$ is {\em swept out by affine spaces} if there exists a dense open set $U\subseteq S$ and a foliation on $U$ such that the Zariski closures of the leaves are affine-linear bi-algebraic subvarieties.
\end{definition}

\begin{thm} \label{prop:swept}
Let $S\subseteq \St$ be an irreducible bi-algebraic variety in a stratum in genus zero. Furthermore, let $W$
 be the smallest torus coset containing $S$ and $V$ the associated affine space. Then $S$ is swept out by affine spaces of dimension \[
\fib(S): = \max \{ \dim A(\{s\}\times V)\,|\, s\in S\}.
\]
In particular, if $\fib(S)>0$, then $S$ is locally in period coordinates defined by a ruled variety.
\end{thm}

The {\em fiber rank} $\fib(S)$ of $S$ is a useful invariant of bi-algebraic varieties.
Bi-algebraic varieties with fiber rank zero exist, in which case \Cref{prop:swept} does not provide any information. For example, consider the {\em locus of exact differentials}
\[\Exc :=\{(X,\omega)\in\St\,|\, \omega \text{ has zero residues}\}\subseteq \St.
\]Restricted to $\Exc$, periods are globally defined algebraic functions, and any algebraic variety contained $\Exc$ is a bi-algebraic variety with fiber rank zero.
Using covering constructions, see \Cref{sec:examples}, we construct bi-algebraic varieties with fiber rank zero outside the locus of exact differentials.
Nonetheless, most bi-algebraic varieties have positive fiber rank. See \Cref{prop:genfib} for a precise statement.

Our final result about the structure of bi-algebraic varieties relates bi-algebraic varieties to residue cosets. Here, we state it for strata with only simple poles. The general statements are \Cref{thm:bialg_structure} and \Cref{prop:simple}.

\begin{cor}\label{cor:rescoset}
Suppose that $\St$ is a stratum with only simple poles and $S\subseteq \St$ an irreducible, bi-algebraic variety.
Then $S$ is a residue coset.

Furthermore, if the residues are constant on $S$, then $S$ is affine-linear.
If $S$ is $\CC^*$-invariant and $\dim R(S) = 1$, then $S$ is linear.

\end{cor}

\subsection{$\overline{\QQ}$-bi-algebraic-geometry  in genus zero and \texorpdfstring{Andr{\'e}-Oort}{}}
The stratum is an algebraic variety over $\ZZ$, so there is a notion of $\overline{\QQ}$-bi-algebraic varieties.
More precisely, if a differential $(X,\omega)$ is defined over a number field, one can ask whether all its periods are algebraic numbers. One must work with the projectivized stratum $\PSt=\St/\CC^*$ for technical reasons.
Bi-algebraic varieties in $\PSt$ correspond to ~$\CC^*$-invariant bi-algebraic varieties in $\St$.

\begin{definition}

A $\overline{\QQ}$-bi-algebraic variety  $S\subseteq \PSt $ is a bi-algebraic variety defined over a number field such that some (equivalently any)  algebraic model in projectivized period coordinates is defined over a number field.
An {\em arithmetic point} $(X,\omega)\in \PSt$ is a $\overline{\QQ}$-bi-algebraic point.
\end{definition}

In \cite[Conj. 2.13]{KlinglerLerer} Klingler and Lerer conjectured an analog of  the Andr{\'e}-Oort conjecture for holomorphic strata. We formulate a version for meromorphic strata,  taking into account residues.

\begin{conj}[Andr{\'e}-Oort in genus zero]\label{conj:andreoort}
    Suppose $S\subseteq\St$ is an algebraic variety containing a Zariski dense of arithmetic points.
    Then $S$ is $\overline{\QQ}$-bi-algebraic.
    Furthermore, if not all residues are identically zero on $S$, then the fibers of the residue map are affine-linear bi-algebraic varieties.
\end{conj}

We prove the Andr{\'e}-Oort conjecture in genus zero in several special cases.

\begin{thm}
\label{thm:andreoort}
Let $S\subseteq \St$ be an irreducible algebraic variety in a stratum in genus zero. Suppose $S$ contains a Zariski dense set of arithmetic points, and one of the following is true
\begin{enumerate}
    \item $S$ is bi-algebraic,
    \item $S$ has constant residues,
    \item $S\subseteq \Exc$ or,
    \item $\fib(S) =0$.
\end{enumerate}
Then the Andr{\'e}-Oort conjecture holds for $S$.
\end{thm}

The primary tool is a characterization of arithmetic points; see \Cref{lemma:arithmetic}. A consequence of the characterization is that arithmetic points are scarce in strata with higher-order poles.
\begin{cor}\label{cor:arithmetic}
Let $\St$ be a stratum in genus zero. Then $\St$ does not contain a Zariski dense set of arithmetic points unless $\St$ has only simple poles, a single pole or a single zero.

\end{cor}
On the other hand, there exist strata with only simple poles that have a Zariski dense set of arithmetic points.

\subsection{Examples of bi-algebraic varieties}
We complement our results by providing a list of geometric constructions of bi-algebraic varieties (not just in genus zero) in \Cref{sec:examples}.

The most basic of these constructions is  {\em residue varieties}, obtained by imposing algebraic equations among residues.
The second basic construction is given by {\em covering constructions}. Starting with a bi-algebraic variety $S$ and fixing a family of covering maps of the underlying curves, one can pull back the differential to produce bi-algebraic varieties in a larger dimensional stratum. 
Following \cite{MoellerMullane}, we define the {\em ``obvious''} bi-algebraic varieties as the intersection of a  residue variety with a covering construction.

We construct infinitely many non-``obvious'' examples of linear varieties. A {\em (meromorphic) Teichm\"uller curve } is a linear bi-algebraic variety of dimension $2$ whose algebraic model is a real $2$-dimensional vector space.
\begin{thm}Let $\St$ be a stratum of genus zero with only simple poles.
There exist infinitely many (meromorphic) Teichm\"uller curves in $\St$ which are not ``obvious".

\end{thm}
We also produce natural examples of affine-linear bi-algebraic varieties that are not linear from log differentials.

\subsection*{Acknowledgements}
I would like to thank Samuel Grushevsky, Leonardo Lerer,  Martin M\"oller, Benjamin Bakker and Scott Mullane  for valuable discussions. I am especially thankful for the argument in \Cref{cor:low} suggested by Leonardo.

\section{Periods in genus zero}\label{sec-periodsetup}
From now on, unless stated otherwise, $\St$ is a stratum of differentials in genus zero. Instead of working locally, period coordinates can be repackaged by passing to the universal cover.
The result is a  {\em bi-algebraic structure} in the sense of Klingler-Ullmo-Yafaev \cite{Bi-algebraic}, given by the following diagram.
\[\begin{tikzcd}
\widetilde{\St} \ar[d,swap,"\pi"]\ar[r,"\phi"]&\calP:= \CC^{\dim \St},\\
\St &
\end{tikzcd}
\]
where
$\pi$ is the universal covering and $\phi$ is the period map obtained by integrating $\omega$ over a global frame for relative homology.

\begin{definition}
An algebraic variety $S\subseteq \St$  is called {\em bi-algebraic} if  the Zariski closure of $\phi(S')$ has dimension $\dim S$ for some analytic irreducible component $S'$ of $\pi^{-1}(S)$.  The Zariski closure $\overline{\phi(S')}^{Zar}$ is called an {\em algebraic model} for $S$.
Similarly we call an algebraic variety $Z\subseteq \calP$ {\em bi-algebraic}, if the Zariski closure of  $\pi(Z')$ has dimension $\dim Z$ for some analytic irreducible component $Z'$ of $\phi^{-1}(Z)$.

\end{definition}

Let \[\mu=(e_0,\ldots,e_{m},-e_{m+1},\ldots,-e_{m+n+1},-2-\sum_{i=0}^{m} e_i- \sum_{j=m+1}^{m+n+1} e_j)\]
a partition of $-2$ with $e_i>0$ for $i=0,\ldots,m+n+1$. The stratum  $\St$ can be identified with  $\CC^*\times \Mgn[0,m+n+3]$, where $\Mgn[0,m+n+3]$ is the moduli space of $m+n+3$-marked rational curves and thus with \[ \CC^*\times \left( \left((\PP^1\setminus \{0,1,\infty\})^{m+n}\right)\setminus \text{small diagonals})\right)\] as follows.
Let $x=(x_1,\ldots,x_m), y=(y_1,\ldots,y_n)$.
To improve the readability of some of the formulas, let $x_0:=0, y_0:=1$.
 Global algebraic  coordinates on the stratum 
are given by 
\begin{equation}
    \label{eq:ratl}
(\lambda,x,y)\mapsto \omega(z)= \lambda\dfrac{\prod_{i=0}^m (z-x_i)^{e_i} }{\prod_{j=0}^n (z-y_j)^{e_{j+m+1}} }dz
\end{equation}
Here $z$ is the standard coordinate on $\PP^1$.

\begin{convention}
From now on, assume that $\mu$ has at least two poles and at least one zero.
After possibly permuting the entries of $\mu$, we can assume that $\omega$ has a zero at $0$ and a pole at  $1$ and $\infty$.
In particular, $m+1$ is the number of zeros, and $n+2$ is the number of poles. 
\end{convention}

There are no residues on strata with a single pole. Thus, in this case, all relative periods are global algebraic functions and every algebraic variety is bi-algebraic.  We thus focus on strata with at least two poles from now on.

\subsection*{Period coordinates}
Besides the global algebraic coordinates given by \cref{eq:ratl}, the stratum $\St$ has local analytic coordinates by integrating $\omega$ over a basis of ~$H^1(X - P(\omega),Z(\omega);\CC)$, where $P(\omega)$ and $Z(\omega)$ denotes the zeros and poles of $\omega$ respectively.
Period coordinates are given explicitly by
\[
\int_{x_0}^{x_j}\omega \text{ for  } j=1,\ldots, m,\quad  2\pi i R_k\text{ for } k=0, \ldots, n,
\]
where $R_k$ is the residue of $\omega$ at the pole $y_k$.
Here $\int_{x_0}^{x_j}\omega $ denotes the integral from $x_0$ to $x_j$ along a chosen path. The choice can only be made consistent in a small neighborhood in $\St$, and the coordinate system depends on the choice. In the sequel, we will always omit the choice of a path.

\begin{prop}\label{prop:Periods} The residues $R_j,\, j=0,\ldots,n$ are algebraic functions on $\St$.
The  relative period $\int_{x_0}^{x_j}\omega $ can locally be computed in terms of the algebraic coordinates $(\lambda,x,y)$ on $\St$ as 
\begin{equation}\label{eq:Per}
\int_{x_0}^{x_j}\omega = F^{alg}_j+F_j^{log},
\end{equation}
where $F_j^{alg}=F^{alg}_j(\lambda,x,y)$ is a rational function and
\[
F_j^{log}(\lambda,x,y) =\lambda\cdot \sum_{k=0}^{n} R_k\log \dfrac{y_k-x_j}{y_k}.
\]
 The branch of logarithm depends on the initial choice of a branch near zero and the path connecting $x_0$ and $x_j$.
\end{prop}

\begin{proof}
Rewrite $\omega$ using partial fractions as
\begin{equation}\label{eq:part}
\omega(z,\lambda,x,y)= Q(z,\lambda,x,y)+\sum_{i=0}^{n}\sum_{j=1}^{e_{m+1+i}} \dfrac{a_{ij}(\lambda,x,y)}{(z-y_i)^j},
\end{equation}
where $a_{ij}$ are rational functions and   $Q$ is a polynomial of degree \[
\sum_{i=0}^m e_i - \sum_{j=0}^n e_{j+m+1}
\]
if $\sum_{i=0}^m e_i \geq  \sum_{j=0}^n e_{j+m+1}$ and $Q\equiv 0$ otherwise.
In particular
\[
R_j=a_{j1}=\Res_{z=y_j} \omega\text{ for $j=0,\ldots,n$.}\]

The formula for periods $\int_{x_0}^{x_j}\omega$ now follows by integrating.

\end{proof}

It will be convenient to separate the algebraic and logarithmic parts of the periods and package different periods together. Let
\begin{equation}
\begin{gathered}\label{eq:periods}
R:= (R_0,\ldots, R_n),\quad F^{alg}:= (F^{alg}_1,\ldots, F^{alg}_m),\quad
F^{log} = (F^{log}_1,\ldots, F^{log}_m),\\
 \alpha:= (F^{alg},2\pi i\cdot R). 
\end{gathered}
\end{equation}

In particular $F^{alg}+F^{log}$ is the vector of relative periods of $\omega$, $R$ the absolute periods and local period coordinates are given by $ (F^{alg}+F^{log}, R)$. The maps $R, F^{alg}, \alpha$ are algebraic functions on $\St$, while $F^{log},P$ are holomorphic functions on the universal covering $\widetilde{\St}$. The map \[
\alpha:\St\to \calP\]
is the {\em algebraic period map.}

\section{Bi-algebraic varieties in strata with few zeros or poles}
 Before delving into the proof of Ax-Schanuel, we discuss bi-algebraic varieties in strata with few zeros and poles.

\begin{prop}
\label{cor:low}
Let $\St$ be a stratum in genus zero and $S\subseteq \St$  a bi-algebraic variety.

\begin{enumerate}
\item
Any algebraic subvariety contained in the locus of exact differentials $\Exc$ is bi-algebraic.

\item 
If $\St$ has a single pole or a single zero, then all algebraic subvarieties are bi-algebraic.
\item
Let $\St$ be a stratum with exactly two poles and  $S\subseteq\St $ bi-algebraic and not contained in the locus of exact differentials.
If the residues are constant on $S$, then $S$ is affine-linear. If $S$ is $\CC^*$-invariant, then $S$ is linear.
\end{enumerate}

\end{prop}
The proof showcases the main ideas going into the general case without having to introduce much notation first. We thank Leonardo Lerer for suggesting the argument in this special case.

\begin{proof}[Proof of \Cref{cor:low}]

On the locus of exact differentials, residues are zero, and in particular, $F^{log}$ is identically zero. Thus, the period map
\[
\phi=(F^{alg},R)
\] is algebraic, and hence every subvariety is bi-algebraic.
If $\St$ has one pole, then the residue is automatically  zero by the residue theorem and thus ~$\St=\Exc$. If $\St$ has only one zero, then there are no relative periods, and  $\phi=R$ is algebraic again.

It remains to address strata with two poles.
We use a slightly different parametrization of $\St\subseteq (\CC^*)^{m+1}$ than \cref{eq:ratl} and choose $0$ and $\infty$ to be poles instead. For example, for strata with only simple zeros, one simple pole, and one higher-order pole, the parametrization is
\[
\omega(z) = \lambda \dfrac{(z-1)(z-x_1)\cdot \ldots \cdot (z-x_m)}{z} dz.
\]

The residue $R$ at $0$ is \[
R  =\pm \lambda x_1 \cdot\ldots \cdot x_m,
\]
where the sign depends on the parity of $m$
and the relative periods are
\[
\int_{1}^{x_i} \omega =F^{alg}_i+ R\log(x_i).
\]
To avoid dealing with branches of logarithms, we will introduce new variables $2\pi iv_j = \log(x_j)$.
Set \[
\Lambda := (\CC^*)^m\times \CC^* \times \CC^m,
\]
with coordinates $(s,v)$, where \[
  s=(x,\lambda)\in(\CC^*)^m\times \CC^*, x=(x_1,\ldots,x_m)\in(\CC^*)^m, v=(v_1,\ldots, v_m)\in\CC^m
 \]
and \[\Gamma:=\{(s,v)\in \Lambda\,|\, e^{2\pi iv_j} = x_j \text{ for } j = 1,\ldots,m\}\]
is the graph of the exponential function\footnote{The factor $2\pi i$ will not play a role until \Cref{sec:Qbar}}.
Define the {\em twisted algebraic period map} $A:\Lambda\to \calP = \CC^{m+1}$ by 
\[
A(s,v) := (F_1^{alg}(x,\lambda)+ 2\pi iRv_1,\ldots, F_m^{alg}(x,\lambda)+ 2\pi iRv_m, 2\pi i R).
\]
Restricted to a component of $(\St \times \CC^m) \cap \Gamma$, the twisted algebraic period map agrees with period coordinates on $\St$. 
Now, suppose $S\subseteq \St$ is bi-algebraic.
Then there exists an analytic irreducible component $S'$ of $(S\times \CC^m) \cap \Gamma$ such that 
\[
\dim Z= \dim S, \text{ where } Z:=  \overline{A(S')}^{Zar}.
\]

Let \[
W = W'\times \CC^*\subseteq (\CC^*)^m\times \CC^*,
\] where $W'$ is the smallest torus coset containing the projection \[\pr_1(S)\subseteq (\CC^*)^m\] and $V$ the associated minimal affine space $V$ with $\exp(V) = W'$. 
In particular, $S'$ is an irreducible component of 
\[
(S \times V)\cap A^{-1}(Z) \cap \Gamma\subseteq S\times V.
\]

\begin{claim}
$A(S\times V)\subseteq  Z$.
\end{claim}
Assume otherwise, then
\[
\codim_{W\times V} A^{-1}(Z) \cap (S\times V) > \codim_{W\times V} S\times V. 
\]
In case $A^{-1}(Z) \cap (S\times V)$ is irreducible, the inequality holds for all irreducible components.
We now compute 
\[
\begin{split}
\codim_{W\times V} S' &= \dim W + \dim V - \dim S\\
&= \codim_{W\times V}  (S\times V) + \codim_{W\times V} (\Gamma \cap( W\times V))\\& <\codim_{W\times V} A^{-1}(Z) \cap (S\times V) +  \codim_{W\times V} (\Gamma \cap (W\times V).
\end{split}
\]
In particular, $S'$ is an unlikely intersection.

By Ax-Schanuel for the exponential function, see \Cref{thm:AxExp}, it follows that $\pr_1(S')=S$ would be contained in a proper torus coset, which contradicts the minimality of $W$.

This proves the claim and thus $Z = \overline {A(S\times V)}^{Zar}$.
In particular, $Z$ contains the affine space
\[
A(\{s\} \times V) =  \{ (F^{alg}(s)+ 2\pi iR(s)v,2\pi iR(s))\,|\, v\in V\}
\]
for any $s$.
Now, there are two cases. Either $R$ is identically zero on $S$, in which case $S$ is contained in the locus of exact differentials. Otherwise 
\[
 \dim A(\{s\} \times V) = \dim V =\dim W' \geq \dim S-1.
\]
If the residues are constant on $S$, then the projection map $S \to \pr_1(S)$ is finite and hence \[
\dim A(\{s\} \times V) = \dim V= \dim S =\dim Z.
\] Thus $Z= A(\{s\}\times V)$  and therefore $S$ is affine-linear.
If $S$ is $\CC^*$-invariant, then $S = \CC^* S_1$, where $S_1$ is bi-algebraic and on $S_1$ the residue $R$ is constant. By the previous argument, $S_1$ is affine-linear and thus $S$ is linear.

\end{proof}

Here, we recall  Ax-Schanuel for the exponential function, first proven by Ax \cite{Ax}, reformulated geometrically. We need a version in families, see  \cite[Cor. A.2]{Pappas} for a reference.

\begin{thm}[Ax-Schanuel in families{ \cite[Cor. 1]{Ax}}]\label{thm:AxExp}
Let \[\exp: \CC^k \to (\CC^*)^k, (v_1,\ldots,v_k)\mapsto (e^{2\pi iv_1},\ldots, e^{2\pi iv_k})\]
and 
$\Lambda:= (\CC^*)^k \times \CC^N \times \CC^k $.
Furthermore, define 
\[
\Gamma:= \{(w,x,v)\in\Lambda \,|\, \exp(v) = w \}.
\]
 to be the graph of the exponential function.
Suppose $Z$ is an algebraic subvariety of $\Lambda $ and $U$ an irreducible analytic component of $Z \cap \Gamma$ such that 
\[
\codim_{\Lambda} U < \codim_{\Lambda} Z +\codim_{\Lambda} \Gamma,
\]
then $\pr_1(U)$ is contained in a proper torus coset, where
\[
\pr_1:\Lambda\to(\CC^*)^k 
\]
is the projection onto the first factor.

\end{thm}

\section{Setup for Ax-Schanuel}\label{sec:setupax}
To deal with the case of general strata one wants to proceed in the same way as for strata with only two poles. As before, to apply the Ax-Schanuel theorem,  one needs to replace each  term $\log \dfrac{y_j-x_i}{y_j}$ in \Cref{prop:Periods} by a new variable. Let
\[
w_{ij}:=  \dfrac{y_j-x_i}{y_j} \text{ for } i=0,\ldots,m,\, j = 0,\ldots,n.
\]
Note that $w_{0j}=1$ and $w_{i0}=1-x_i$. For brevity  set 
\[
w= (w_1,\ldots,w_m), w_i=(w_{i0},\ldots,w_{in}) \text{ for } i=1,\ldots,m.
\]
The stratum can be embedded into the new coordinate system using the locally closed embedding
\[
\begin{split}
\iota:\St &\to (\CC^*)^{m(n+1)}\times \CC^*,\\
(\lambda,x,y)&\mapsto (w,\lambda)
\end{split}
\]

The closure of $\iota(\St)$ is defined by the relations
\[
(w_{i0}-1)(w_{i'j}-1) =  (w_{i'0}-1)(w_{ij}-1)
\]
for  $i,i'=1,\ldots,m, j=0,\ldots,n$.

Let $k:=m(n+1)$ and consider
\[
\Lambda:= \{(s,v) \in (\CC^*)^k\times \CC^*\times \CC^k \}
\]
with coordinates $(s,v)$, where 
\begin{gather*}
    v=(v_1,\ldots,v_m)\in\CC^k , v_i=(v_{i0},\ldots, v_{in})\in\CC^{n+1},\\ 
     w= (w_1,\ldots,w_m)\in(\CC^*)^k , w_i=(w_{i0},\ldots, w_{in})\in(\CC^*)^{n+1},\\
     s=(w,\lambda)\in  (\CC^*)^k\times \CC^*.
\end{gather*}

Define \[
\begin{gathered} x= (x_0,x_1,\ldots,x_m),\text{ where $x_0=0,\, x_i= 1-w_{i0}$ for $i=1,\ldots,m$,}\\
y= (y_0,\ldots,y_n),\text{ where $y_0=1,\, y_j= 1-\dfrac{x_1}{w_{1j}}$ for $j=1,\ldots,n$,}
\end{gathered}
\]

By replacing   $\log w_{kl}\mapsto 2\pi iv_{kl}$, period coordinates can be transferred to  $ \Lambda $. Thus in analogy to  \cref{eq:periods}  define
\begin{equation}\label{eq:A}
\begin{gathered}
R(s,v):= R(\lambda,x,y),\quad 
F^{alg}(s,v) = F^{alg}(\lambda,x,y),\\
T_R(v):= 2\pi i(R\cdot v_1,\ldots, R\cdot v_m),\quad
A(s,v):= (F^{alg}(s)+T_{R(s)}(v),2\pi i R(s)).
\end{gathered}
\end{equation}

If one plugs in  $2\pi iv_{kl} = \log w_{kl}$, then $T_R(v)= F^{log}$ and $A(s,v)$ coincides with period coordinates. We refer to $A$ as {\em the twisted algebraic period map}. Let \[
\Gamma = \{(s,v)\in\Lambda\,|\, w_{ij} = \exp(v_{ij} )\text{ for all } i,j\} \] be the graph of the exponential map.

When restricted to an analytic irreducible component  $(\St\times \CC^k)\cap \Gamma$
  the twisted algebraic period map $A$ agrees with a branch of period coordinates on $\St$.
  
In this new language, an algebraic subvariety $S\subseteq \St$ is bi-algebraic if and only if there exists an analytic component $S'$ of $(S\times \CC^k) \cap\Gamma$  such that \[
\dim \overline{A(S')}^{Zar} = \dim S,\]
where $\overline{A(S')}^{Zar}$ is the Zariski closure of $A(S')$.

\begin{definition}
A {\em torus coset} $W\subseteq \St$ is an irreducible component of \[
(W'\times \CC^*) \cap \St,
\]
where  $W'\subseteq (\CC^*)^k$ is a translate of a subtorus.
The {\em associated translated rational subspace } $V$ is defined by the property that $\exp(V) = W'.$
If $W'$ is the smallest torus coset containing $\pr_1(S)\subseteq (\CC^*)^k$, then $V$ is the {\em minimal affine space associated } with $S$.
In that case we call any irreducible component of $(W'\times \CC^*) \cap \St$ a minimal torus coset containing $S$.

An irreducible, algebraic variety $S\subseteq \St$ is a {\em residue coset } if\[
S\eqc  W\cap R^{-1}(Z)\] for a torus coset $W$ and a subvariety $Z\subseteq \mathcal{R}$.
Here and in the sequel $\eqc$ means that $S$ is an irreducible component $ W\cap R^{-1}(Z)$. We only use this notation if the space on the left-hand side is irreducible.
\end{definition}

In \Cref{tab:notation} we summarize the notations that will be used frequently in the sequel.

\begin{table}[H]
\caption{Notation for  period maps}
\begin{tabularx}{\textwidth}{@{}XX@{}}
\toprule
  Notation & \\
  $m$ & Number of zeros of $\St$ -1 \\
  $n$ & Number of of poles of $\St$ - 2 \\
  $k=m(n+1)$ & \\
  $\calP=\CC^{m+n+1}$ & Period space \\
  $\calR=\CC^{n+1}$ & Residue space\\
 
  $R: \St\to \calR $ & Residue map\\
   $F^{alg}:\St\to \CC^{m}$ & algebraic part of a relative period\\
  $F^{log}:\widetilde{\calH(\mu)}\to \CC^m$ & logarithmic part of a relative period\\
  $\alpha:\St\to \calP$ & algebraic period map\\
 $\alpha(s) = (F^{alg}(s), \,2\pi iR(s)), s\in \St$ &\\ 
  $\logSt = \{s\in \St\,|\, F^{alg}(s) = 0\}$&\\
   $\Lambda = (\CC^*)^k\times\CC^*\times \CC^k$ & \\
    $\Gamma\subseteq \Lambda$ & Graph of the exponential function \\
   $A:\Lambda\to \calP$ & twisted algebraic period map\\
\bottomrule
\end{tabularx}
\label{tab:notation}
\end{table}

\section{Ax-Schanuel in genus zero} The goal of  this section is to prove Ax-Schanuel, \Cref{thm:ax}.

\begin{definition}
Let $S\subseteq \St$ be an irreducible, algebraic subvariety and $V$ the associated minimal affine space. The {\em fiber rank of $S$} is 
\[
 \fib(S) :=  \max \{ \dim A(\{s\}\times V)\,|\, s\in S \}  =\max \{ \dim T_{R(s)}V\,|\, s\in S \},
\]
where we recall 
\[
\begin{gathered}
T_R(v)= 2\pi i(R\cdot v_1,\ldots, R\cdot v_m),\\
\text{ where } v=(v_1,\ldots, v_m)\in \CC^{k},  v_i\in \CC^{n+1}, i =1,\ldots,m, \, R\in \CC^{n+1}.
\end{gathered}
\]
\end{definition}

\begin{remark}
The fiber rank of $S$ only depends on $R(S)$ and the smallest torus coset $W$ containing $S$. In particular if $S'$ is an irreducible component of \[
R^{-1}(\overline{R(S)}^{Zar}) \cap W\] containing $S$, then  $S'$ is minimal residue coset containing $S$ and furthermore
\[
\dim R(S)=\dim R(S'), \, \fib(S) = \fib (S').
\]
\end{remark}

We now set up the bi-algebraic structure for Ax-Schanuel.
Let
\[
\begin{tikzcd}
\widetilde{S} \ar[r,"\phi"]\ar[d,"\pi",swap]& \calP& & S\times V \ar[d,"\rho"]\\
S & & \widetilde{S} \ar[r,"\pi\times \phi ",swap]& S\times \calP,
\end{tikzcd}
\]
where
\begin{itemize}
\item $V$ is the minimal affine space associated with $S$,
\item $\pi:\widetilde{S}\to S$ is the universal covering,
\item $\phi: \widetilde{S}\to A(S\times V)\subseteq\calP$ is the period map,
\item $\rho:S\times V\to S\times \calP, \,\rho(s,v) = (s, A(s,v))$.
\end{itemize}

Let \[
\Omega_S=\rho(S\times V)\subseteq S\times \calP\]
and $\Sigma_S\subseteq \Omega_S$ be the graph of $\pi\times \phi:\widetilde{S}\to S\times \calP$.
Then 
\[
\dim \Omega_S = \dim S  + \fib(S), \quad \codim_{\Omega_S} \Sigma_S= \fib(S).
\]

\begin{thm}[Ax-Schanuel in genus zero]
\label{thm:ax-residue}
Let $\St$ be a stratum of genus zero differentials and $S\subseteq  \St$ be a residue coset. Furthermore, let $Z\subseteq \Omega_S$ be an algebraic variety and $U$ an analytic irreducible component of $Z\cap \Sigma_S$ such that 
\[
\codim_{\Omega_S} U < \codim_{\Omega_S} Z + \codim_{\Omega_S} \Sigma_S,
\]
then $\pr_1(U)$ is contained in a proper residue coset $S'$, where $\pr_1:\Omega_S\to S$ is the projection.

\end{thm}

\begin{proof}
Suppose $\pr_1(U)$ is not contained in a proper residue coset.
We will compare the bi-algebraic structure of $S$ to the bi-algebraic structure of $V$ coming from the exponential map and then apply Ax-Schanuel for $\exp$.

The following diagram commutes, where the horizontal arrows are inclusions.
\[
\begin{tikzcd}
 \rho^{-1}(Z)\cap  \Gamma \ar[r] \ar[d,"\rho"]& \rho^{-1}(Z) \ar[r]\ar[d,"\rho"]&  S\times V\ar[d,"\rho"]\\
 U \ar[r] &  Z\ar[r] & 
\Omega_S
\end{tikzcd}
\]
Since the period map is a local biholomorphism, one has
\[
\dim U = \dim \rho^{-1}(Z)\cap \Gamma,
\]
from which it follows
\[
\begin{split}
\codim_{\Omega_S} U &=  \dim {\Omega_S}- \dim U \\
&= \dim {\Omega_S}- \dim \rho^{-1}(Z)\cap  \Gamma\\
&= \dim {\Omega_S}-\dim \rho^{-1}(Z) + \codim_{W\times V} \Gamma\\
&=\dim {\Omega_S}- \dim Z - \dim F  + \dim V \\
& = \codim_{\Omega_S} Z + \codim_{\Omega_S} \Sigma_S  -\fib(S) - \dim F +\dim V  ,
\end{split}
\]
where $F$ is the  general fiber of  $\rho_{|\rho^{-1}(Z)}$. Note that for the fourth equality we used that $\rho^{-1}(Z)$ and $\Gamma$ intersect generically transversely, since otherwise \[
\pr(\rho^{-1}(Z)\cap \Gamma) = \pr_1(U)
\] would be contained in a proper torus coset by Ax-Schanuel for the exponential map, where  $\pr:  S\times V\to S$ is the projection onto the first factor. This would contradict the minimality of $V$.
The quantity $\dim V - \fib(S)$ is the dimension of the general fiber $F_S$ of $\rho$. Thus, it suffices to show that the general fiber dimension of $\rho$ and of $\rho_{|\rho^{-1}(Z)}$ agree.
For any $(s,x)\in \Omega_S$ the fiber of $\rho$ over $(s,x)$ has dimension
\[
\dim  \ker T_{R(s)}(V).
\]
Thus,  $\dim F$ only depends on the image $R(\pr_1(Z))$.
If ~$\dim F>\dim F_S$, then $R(\pr_1(Z))$ would be contained in the proper subvariety 
\[
R' := \{ r \in \overline{R(S)}^{Zar}\,|\, \dim \ker T_{r}(V) > \dim F_S.\}
\]
Hence, $\pr_1(U)$ is contained in $R^{-1}(R')$ and thus would be contained in a proper residue coset.

 \end{proof}

\begin{remark}
It follows from the proof that, if $S$ is a minimal residue coset containing $\pr_1(U)$, then 
\[
\codim_{\Omega_S} U = \codim_{\Omega_S} Z + \codim_{\Omega_S} \Sigma_S.
\]
Since $\Omega_S$ might not be smooth in general, this does not seem to follow directly from general principles.
\end{remark}

The following corollary establishes the implications $(1)\Leftrightarrow (2)$ of \Cref{thm-characterization}.
\begin{cor}\label{cor:bi}
Let $S\subseteq \St$ be an irreducible, algebraic variety.
Then $S$ is bi-algebraic if and only if \[
\dim S = \dim A(S\times V),
\]
where $V$ is the minimal affine space associated with $S$. In that case, for any irreducible analytic component $S'$ of $\pi^{-1}(S)$ one has \[
\overline{\phi(S')}^{Zar} = \overline{A(S\times V})^{Zar}.
\]
where we recall that \[
\pi:\widetilde{\St}\to \St
\]
is the universal covering and \[
\phi:\widetilde{\St}\to \calP\]
the period map.
\end{cor}

\begin{proof}
First, suppose \[
\dim S = \dim A(S
\times V).
\]
Let $S'$ be a component of $\pi^{-1}(S)$.
Then 
\[
\dim S = \dim \phi(S')\leq \dim \overline{\phi(S')}^{Zar} \leq \dim A(S\times V) = \dim S.
\]
Thus, $S$ is bi-algebraic.

Now assume that $S$ is bi-algebraic. Let $S'$ be a component of $\pi^{-1}(S)$ and $Z=\overline{\phi(S')}^{Zar}$.
 Let $Y$ be a minimal residue coset containing $S$. In particular $\fib(Y) = \fib(S)$.
Apply Ax-Schanuel to $(S\times Z) \cap \Omega_Y$.
Then 
\[
\codim_{\Omega_Y} \phi(S') = \codim_{\Omega_Y} (S\times Z) \cap \Omega_Y + \codim_{\Omega_Y} \Sigma_Y
\]
and hence 
\[
\dim (S\times Z) \cap \Omega_Y = \dim S + \fib(Y) = \dim \Omega_S.
\]
In particular $S\times Z$ contains $\Omega_S$ and thus
$\{s\}\times Z \supseteq \{s\} \times A(\{s\}\times V)$ for any $s$. It follows that $Z$ contains $A(S\times V)$ and 
\[
\dim A(S\times V) \geq \dim S = \dim Z \geq \dim A(S\times V),
\]
hence $\dim S = \dim A(S\times V)$.

\end{proof}

\begin{remark}
Note that in general, if $S\subseteq \St$ is bi-algebraic, then for every component $W$ of $\pi^{-1}(S)$  the Zariski closure of the image $\phi(W)$ is an algebraic variety $Z_W$ of dimension $\dim S$ and the monodromy permutes the different varieties $Z_W$.
The above theorem says that in genus zero, $Z_W$  is already monodromy-invariant for any $W$, i.e., all components are mapped to the same image in period coordinates. In particular, every bi-algebraic variety has a unique algebraic model $\overline{A(S\times V)}^{Zar}$.
\end{remark}

We collect some useful remarks about the fiber rank and bi-algebraic varieties.
\begin{lemma}\label{lemma:fib}
Let $S\subseteq \St$ be an irreducible algebraic subvariety in a stratum of genus zero with minimal affine space $V$.
\begin{enumerate}
\item The inequalities  \[\fib (S)\leq \dim A(S\times V),\quad \dim S \leq \dim A(S\times V)\] hold.
\item If $S$ is bi-algebraic, then $\fib(S)\leq \dim (S)$. In case of equality, $S$ is affine-linear with algebraic model $A(\{s\}\times V)$ for generic $s\in S$.
\item If $\fib(S)=0,$ then $S$ is bi-algebraic.
\end{enumerate}

\end{lemma}

\begin{proof}
The inequality $\fib(S)\leq \dim A(S\times V)$ follows from 
\[
\dim A(\{s\}\times V)\leq \dim A(S\times V)
\]
for all $s\in S$ and the inequality $\dim S\leq \dim A(S\times V)$ follows since period coordinates are local coordinates. This shows (1).
If $S$ is bi-algebraic, then $\dim S =\dim A(S\times V)$, hence the inequality $\fib (S)\leq \dim S$ follows from $(1)$.
In case of equality $\fib(S) = \dim S$, it follows that \[
A(\{s\}\times V)= A(S\times V)\]
for generic $s$ and $A(S\times V)$ is the algebraic model in period coordinates. Furthermore, $A(\{s\}\times V)$ is affine-linear by the definition of the twisted algebraic period map $A$.

Finally, if $\fib(S) =0$, then
\[
\dim \alpha(S) \leq \dim S \leq \dim A(S\times V) = \dim \alpha(S)
\]
and $S$ is bi-algebraic by \Cref{cor:bi}.

\end{proof}

The following result shows that a bi-algebraic variety $S$ can be recovered from a minimal torus coset containing $S$ as well as the image $\alpha(S)$ under the algebraic period map.

\begin{thm}[Structure theorem for bi-algebraic varieties]\label{thm:bialg_structure}
Let $S\subseteq \St$ be an irreducible algebraic variety in a stratum in genus zero.
If $S$ is bi-algebraic,
 then \[
S \eqc W \cap \alpha^{-1}(Z),
\]
where 
\begin{itemize}
    \item  $W$ is a minimal torus coset containing $S$,
    \item  $Z=\overline{\alpha(S)}^{Zar}.$
\end{itemize}
\end{thm}

\begin{proof}
First, assume $S$ is bi-algebraic.
Let $Y$ be an irreducible component of
\[W\cap \alpha^{-1}(\overline{\alpha(S)}^{Zar}).
\]
containing $S$. By definition, $S\subseteq Y$ and the goal is to show $Y\subseteq S$.

We claim that
\[
 A(Y\times V )=A(S\times V)
\]
from which it follows that $\dim Y\leq \dim A(Y\times V ) = \dim A(S\times V)=\dim S.$
Hence $Y$ and $S$ agree.
The see the claim let $y\in Y$ be a generic point.
There exist $s\in S$ such that
\[
R(y)=R(s),\, F^{alg}(y) = F^{alg}(s)
\]
and thus
\[
A(y,v) = (F^{alg}(s)+ T_{R(s)}(v),2\pi i R(s))\in A(S\times V).
\]

\end{proof}

\subsection{Bi-algebraic geometry of $\logSt$}
We now focus on bi-algebraic varieties contained in the {\em log stratum} \[
    \logSt: = \{s\in\St\,|\, F^{alg}(s) =0\}.
    \]

The main examples of bi-algebraic varieties in log strata are bi-algebraic varieties in strata with simple poles and, as we will see in \Cref{sec:Qbar}, bi-algebraic varieties containing a Zariski dense of arithmetic points.

For any variety $S\subseteq \logSt$ one has 
\[
\alpha(s) = (0,2\pi i R(s)) \text{ for } s\in S.
\]
and
\begin{equation}\label{eq:fibdim}
\dim A(S\times V) = \dim R(S) + \fib(S).
\end{equation}

By specializing  \Cref{thm:bialg_structure} to bi-algebraic varieties in $\logSt$, one obtains the following.

\begin{thm}[Structure theorem for bi-algebraic varieties in $\logSt$]\label{prop:simple}
Let $\St$ be a stratum in genus zero and $S\subseteq\logSt\subseteq\St$ an irreducible, algebraic subvariety.
If $S$ is bi-algebraic, then $S$ is a component of 
\[W\cap R^{-1}(Z) \cap \logSt\] and \[
\dim S = \dim R(S) + \fib(S),
\]
where \begin{itemize}
    \item  $W\subseteq \St$ is a minimal torus coset containing $S$,
    \item $Z= \overline{R(S)}^{Zar}\subseteq \calR$.
    \end{itemize}
Furthermore, the fibers of the  residue map \[R_{|S}:S\to\calR\] are affine-linear bi-algebraic varieties. Thus, if the residues are constant on $S$, then $S$ is affine-linear.
Conversely, if $S$ is a component of  \[
W\cap R^{-1}(Z)\cap\logSt,
\]
where $W\subseteq\St$ is a torus coset and $Z\subseteq \calR$ an algebraic variety and furthermore satisfies \[
\dim S = \dim Z + \fib(S),
\] then $S$ is bi-algebraic.

\end{thm}

\begin{proof}
This now follows from \Cref{thm:bialg_structure} combined with \Cref{eq:fibdim}.
To see that the fibers of the residue map are affine-linear, note that any irreducible component of $S \cap R^{-1}(\{R(s)\})$ has algebraic model
\[
 A(\{s\}\times V) = (T_{R(s)}V, 2\pi iR(s)).
\]
\end{proof}

By further specializing to strata with only simple poles, we obtain  \Cref{cor:rescoset}.

\begin{proof}[Proof of \Cref{cor:rescoset}]
    For strata with simple poles, one has $\logSt=\St$. If, additionally, the residues are constant, then $S$ is a fiber of the residue map $R_{|S}:S\to \calR$. The first part thus follows from \Cref{prop:simple}.
    For the second case, write $S= \CC^* S_1$, where $S_1$ has constant residues, and apply the first part.
\end{proof}

\section{Foliations with affine-linear leaves}
In this section, we finish the proof of \Cref{thm-characterization}.
Thus far we have established $(1)\Leftrightarrow (2)$. The proof of the remaining implications is split into two parts. The following lemma shows the implication $(3)\Rightarrow (2)$.
\begin{lemma}\label{lemma:swept} Let $S\subseteq \St$ is an algebraic variety and $V$ the minimal affine space associated with $S$.
Suppose for any $s\in S$ there exists an affine-linear bi-algebraic subvariety $X_s\subseteq S$ containing $s$ with algebraic model $A(\{s\}\times V)$, then $S$ is bi-algebraic.
\end{lemma}

\begin{proof}
Our goal is  to show 
\[
\dim A(S\times V) = \dim S.
\]

For each $s\in S$ let $X_s\subseteq S$ be a bi-algebraic subvariety with minimal affine space $V_s$ such that 
\[
A(X_s\times V_s) = A(\{s\}\times V).
\]

Let $U\subseteq \St$ be a small period chart containing a generic point $s\in S$ and recall that $\phi:U\to \calP$ is the period map.

Then
\[
\begin{gathered}
\phi(S\cap U) = \cup_{s\in U} \phi((X_s\cap U) =  \cup_{s\in S} A(X_s \times V_s) \cap \phi(U)=  \\
\cup_{s\in S} A(\{s\} \times V) \cap \phi(U)= A(S\times V) \cap \phi(U).
\end{gathered}
\]
Hence $\dim S = \dim A(S\times V)$ and thus $S$ is bi-algebraic.
\end{proof}

 The next proposition proves  the implication $(3)\Rightarrow (1)$ in \Cref{thm-characterization}.

\begin{prop}\label{prop:foliation}
Let $S\subseteq \St$ be bi-algebraic with minimal affine space $V$.
Then $S$ is swept out by affine-linear bi-algebraic varieties. More precisely, over a dense open set, there exists a foliation with leaves 
\[
 X_s \eqc \overline{\pr_1(A_{|S\times V}^{-1}(A(\{s\}\times V))}^{Zar}
\]

and $X_s\subseteq S$ is bi-algebraic with algebraic model $A(\{s\}\times V)$. In particular   \[
\dim X_s = \fib(S)\]
for generic $s\in S$.

\end{prop}

\begin{proof}
Let $W$ be a minimal torus coset containing $S$ with associated translated rational subspace $V$.
For every $s$ let
\[
\hat{X}_s:=\overline{\pr_1(A_{|S\times V}^{-1}(A(\{s\}\times V))}^{Zar}
\]
and $X_s$ be the connected component of $\hat{X}_s$ containing $s.$

The first step is to show that $\hat{X}_s$ is bi-algebraic with algebraic model $A(\{s\}\times V)\subseteq A(S\times V)$. It then follows that $\hat{X}_s$ is smooth and thus  $X_s$ is irreducible.  Afterward, we need to ensure that the tangent spaces of $X_s$ vary holomorphically with $s$.

\begin{claim} $\hat{X}_s$ is bi-algebraic.
\end{claim}

Let $s\in S$ 
and $U$ be a small open set intersecting $\hat{X}_s$ such that the period map  $\phi:U\to\calP$ is a biholomorphism onto its image.
 Let $S'$ be a component of $\left(S\cap U)\times V\right) \cap \Gamma$. In particular, $S'$ corresponds to a choice of branch of logarithm on $U$ and since $S$ is bi-algebraic, it follows 
 \[
 A(S') = \phi(U) \cap A(S\times V).
 \]
 Suppose $s'\in \hat{X}_s$ is a general point and $(s',v')\in S'$. It follows from the definition of $\hat{X}_s$ that there exists $v_0\in V$ such that $A(s',v_0) = A(s,v)$ for some $v\in V$, i.e.,
 \[
  (F^{alg}(s')+T_{R(s')}v_0,2\pi i R(s')) = (F^{alg}(s)+T_{R(s)}v,2\pi iR(s)). 
 \]
 Thus 
 \[
 A(s',v') = (F^{alg}(s)+T_{R(s)}(v'+v-v_0),2\pi i R(s))\in A(\{s\}\times V).
 \]

 We conclude that 
 \[
 A(S' \cap (\hat{X}_s\times V)) \subseteq A(\{s\}\times V).
 \]

 On the other hand if $(s',v')\in S' \cap A^{-1}(\{s\}\times V)$ then by definition $s'\in \hat{X}_s$.
 Thus,  $\hat{X}_s$ agrees in period coordinates with $A(\{s\}\times V)$.
 Since $A(\{s\}\times V)$ is affine-linear, $\hat{X}_s$ is smooth and thus the connected component $X_s$ is an irreducible component.

It remains to construct a foliation.
Let $U$ be the open dense set where $S$ is smooth and $\dim X_s=\fib(S)$.

The algebraic model of $X_s$ is \[
A(\{s\}\times V) = \{ F^{alg}(s)+ T_{R(s)}V,2\pi i R(s))\,|\, v\in V\},
\]
which is an affine-linear space with tangent space $T_{R(s)}V^{lin}$.
We define a subbundle $\mathcal{F}$ of the tangent bundle $\calT_U$ via $\mathcal{F}_s = T_{R(s)}V^{lin}$, which varies holomorphically with $s$, since $R$ is algebraic.
By definition, $X_s$ is a leaf of the foliation.
\end{proof}

We have now finished the proof of \Cref{thm-characterization}. For the convenience of the reader, we give references to the different statements.

\begin{proof}[Proof of \Cref{thm-characterization}]
\Cref{cor:bi} proves the implication $(1)\Leftrightarrow (2)$. From \Cref{lemma:swept} it follows that $(3)\Rightarrow (1)$ and the final implication $(1)\Rightarrow (3)$ is part of \Cref{prop:foliation}.
\end{proof}

Additionally, \Cref{prop:foliation} establishes the main part of \Cref{prop:swept}. We can now finish the proof.
\begin{proof}[Proof of \Cref{prop:swept}] 
It remains to show that if $S$ is any bi-algebraic variety with positive fiber dimension, then the algebraic model is a ruled variety.
Since $\fib S>0$, the algebraic model of $S$ contains positive dimensional affine spaces through every point.
More precisely, if $V$ is the minimal affine space associated with $S$, then $A(\{s\}\times V)$ is affine-linear of positive dimension.
On a dense open set $U\subseteq S$, the dimension of $A(\{s\}\times V)$  is constant.
The ruling is provided by considering the image of the map 
\[
s\mapsto \CC\cdot A(\{s\}\times V)
\]
from $U$ into the Grassmannian of $\fib(S)+1$-planes in $\calP = \CC^{\dim \St}$. Here, we used that if $U$ is small enough, then  $U$ is not in the locus of exact differentials. In particular, $A\{s\}\times V$ is affine, but not linear, of dimension $\fib(S)$.  

\end{proof}

\begin{remark}
In the setup of  \Cref{thm-characterization},
if $S$ is a $\CC^*$-invariant bi-algebraic variety, $S$ descends to the projectivized stratum $\PP\St$.  It follows from the proof of \Cref{prop:swept} that, if $S\subsetneq \Exc$,  then  $S$ is swept out by linear spaces of dimension $\fib(S)+1$. Thus, the image in the projectivized stratum is ruled as soon as $\fib(S)>0$.
\end{remark}

If the fiber rank of a bi-algebraic variety $S$ is zero, then 
\Cref{prop:foliation} does not impose any restrictions on $S$. On the other hand, if the residues are generic, i.e., not contained in a proper $\QQ$-vector space, then the foliation is always non-trivial.

\begin{prop}\label{prop:genfib}
Let $S\subseteq \St$ be a positive-dimensional bi-algebraic variety in genus zero. Assume that $\St$ has at least two zeros and that the residues are generically $\QQ$-linearly independent on $S$. Then $S$ has positive fiber rank. In particular, any algebraic model of $S$ is a ruled variety.
\end{prop}

\begin{proof}
Let $S$ be a bi-algebraic variety such that the residues are generically $\QQ$-linearly independent. In particular, $S$ is not in the locus of exact differentials.

Recall that in \Cref{sec-periodsetup} we embedded $\St$ in $(\CC^*)^{m(n+1)} \times \CC^*$.
Assume that  $\St$ has at least two zeros, then $m>0$. Let $W'$ be the smallest torus coset containing $\pr_1(S)$ with associated affine space $V$.
As before, let $W=W'\times \CC^*$.
Suppose $S$  is a bi-algebraic variety with fiber rank zero, i.e.,
\[
T_{R(s)}V^{lin} =\{0\} \text{ for all } s\in S,
\]
where $V^{lin}$ is the linear subspace of the affine space $V$.
There are two cases to consider. First assume $ \dim W' = \dim V >0$. In that case let $e=(e_1,\ldots,e_m)\in V^{lin}(\QQ), e_i\in \QQ^{n+1}$ be a non-zero rational point.  Then  $R(s)$ is contained in the proper $\QQ$-subspace
\[
\{ r \in \calR\,|\, r\cdot e_1 = \ldots =r\cdot e_m = 0\}\subsetneq \calR
\]
for any $s\in S,$
contradicting that the residues are $\QQ$-linearly independent.

The remaining case is $\dim W'=0$, i.e., $\pr_1(S)$ is a point.
This implies that 
$
w_{ij} = \dfrac{y_j-x_i}{y_j}
$
is constant on $S$ for all $i,j$.
In particular  $w_{i0} = 1-x_i$ is constant. Then, it follows that the functions $y_j$ are constant for all $j$. It follows that $S$ has dimension zero.

\end{proof}

\section{\texorpdfstring{$\overline{\QQ}$-bi-algebraic geometry in genus zero}{}} \label{sec:Qbar} For the rest of this section, $\St$ is a stratum in genus zero. We now turn from general bi-algebraic varieties to $\overline{\QQ}$-bi-algebraic geometry. Recall that a differential $(X,\omega)\in \St$ defined over $\overline{\QQ}$ is an arithmetic point if
\[
\left[ \int_{x_0}^{x_1} \omega:\ldots:\int_{x_0}^{x_m} \omega: 2\pi i R_0:\ldots 2\pi i R_n\right]\in \PP^{m+n+1}(\overline{\QQ}).
\]

\begin{lemma}[Structure of arithmetic points]
\label{lemma:arithmetic}
Let $(X,\omega)\in\St $ be defined over $\overline{\QQ}$.
Then $(X,[\omega])\in\PSt)$ is arithmetic if and only if 
 one of the following is true.
 \begin{enumerate}
\item Either $R=0$ or
\item $F^{alg}=0,\, F^{log}_j\in 2\pi i\langle R_0,\ldots,R_n\rangle_{\QQ}$ for all $j$.

 \end{enumerate}
 In particular, either $(X,\omega)$ is in the log stratum $\logSt$ or an exact differential. 

\end{lemma}

\begin{proof}
Suppose that $(X,\omega)$ is arithmetic. Then there exists $\alpha\in\CC^*$ such that
\[
\alpha \cdot\int_{x_0}^{x_j}\omega, \alpha\cdot2\pi iR_l \in \overline{\QQ}
\]
for all $j,l$.
First, suppose $(X,\omega)$ is not in the locus of exact differentials. After renumbering, we can assume $R_0\neq 0$, and since $\omega$ is defined over  $\overline{\QQ}$, the residue $R_0$ is an algebraic number.
 In particular  $\alpha = \dfrac{1}{2\pi i}q$ for some $q\in \overline{\QQ}\setminus \{0\}$ and thus
\begin{equation*}
 \dfrac{1}{2\pi i}q\int_{x_0}^{x_j} \omega = \dfrac{1}{2\pi i}q\left(F^{alg}_j(x,y,\lambda)+\sum_{l=0}^m R_l \log w_{jl}\right)\in\overline{\QQ} \text{ for any j}.
\end{equation*}

By Baker's theorem on  linear forms in logarithms, see \Cref{thm:baker}, it follows that $F^{alg}=0$ and either
all coefficients $R_l, l=0,\ldots, n$ are zero or
there has to be a $\QQ$-linear relation among
\[
\log w_{j1},\ldots \log w_{jl}, 2\pi i.
\]
After rewriting some $\log w_{jl}$ as a $\QQ$-linear combination of the remaining terms, one can apply Baker's theorem again. 
One then proceeds inductively. At each step, one can remove a term $\log w_{kl}$ until either the whole expression is zero or only one term $\log w_{kl}$ is left. In the latter case, $\log w_{jl} \in 2\pi i \QQ$. 
This implies (2).

It remains to address the case that $S\in\Exc$.
In this case it follows that $R=0$ and $F^{alg}_j \in \overline{\QQ}$ for all $j$.  In particular any $(X,\omega)$ defined over $\overline{\QQ}$ is arithmetic.

The converse direction follows from the definition of an arithmetic point.
\end{proof}
\begin{thm}[Linear forms in logarithms, \texorpdfstring{\cite[Thm 2.1]{Baker}}{}]\label{thm:baker}

Suppose \[\alpha_1,\ldots,\alpha_n\in\overline{\QQ}\setminus \{0\}
\] are non-zero algebraic integers such that \[
\log \alpha_1,\ldots,\log \alpha_n\]
are $\QQ$-linearly independent. Then $1, \log \alpha_1,\ldots,\log \alpha_n$ are $\overline{\QQ}$-linearly independent.
\end{thm}

As a consequence, every arithmetic point is contained in  $\logSt\cup \Exc$. For strata with only  simple poles, a single pole, or a single zero, $\logSt\cup \Exc$ coincides with the whole stratum. For any other stratum, $\logSt\cup \Exc$ is always a proper subvariety, and  arithmetic points are never Zariski dense. This proves \Cref{cor:arithmetic}.

We now turn to the Andr{\'e}-Oort conjecture \Cref{conj:andreoort} in genus zero. 
 The second part of the conjecture follows from the structure theorem for bi-algebraic varieties.
\begin{cor}Suppose $S\subseteq \St$ is bi-algebraic and contains a Zariski-dense set of arithmetic points. If not all residues are identically zero on $S$. Then the fibers of the residue map are affine-linear bi-algebraic varieties.

\end{cor}

\begin{proof}
Since not all residues vanish identically on $S$, the variety $S$ is not contained in $\Exc$. By \Cref{lemma:arithmetic}, $S$ is contained in the log stratum $\logSt$. Hence, the structure theorem for bi-algebraic varieties in $\logSt$ \Cref{prop:simple} implies that the fibers of the residue map are affine-linear.

\end{proof}

We now turn to the proof of \Cref{thm:andreoort}.
For the convenience of the reader, we state it again.
\begin{thm}[= \Cref{thm:andreoort}]
\label{thm:andreoort2}
Let $S\subseteq \St$ be an irreducible algebraic variety in a stratum in genus zero. Suppose $S$ contains a Zariski dense set of arithmetic points, and one of the following is true
\begin{enumerate}
    \item $S$ is bi-algebraic,
    \item $S$ has constant residues,
    \item $S\subseteq \Exc$ or,
    \item $\fib(S) =0$.
\end{enumerate}
Then $S$ is $\overline{\QQ}$-bi-algebraic.
\end{thm}


    

\begin{proof}
We start with  (3). Note that $S$ is defined over $\overline{\QQ}$ since it contains a Zariski dense set of $\overline{\QQ}$-rational points. Every subvariety $S\subseteq \Exc$ defined  over $\overline{\QQ}$ is $\overline{\QQ}$-bi-algebraic. 

For the rest of the proof, assume $S$ is not contained in $\Exc$.
We start with the proof of (1). Let $V$ be the minimal affine space associated with $S$ and $W$ a minimal torus coset containing $S$.
Since $S$ is bi-algebraic, 
\[
\dim A(S\times V) = \dim S.
\]
It suffices to show that $\overline{A(S\times V)}^{Zar}$ is defined over $2\pi i\overline{\QQ}$.

Let $X = \{p_j, j\in j\}\subseteq S$ be a Zariski dense set of arithmetic points.
For every $p_j$ there exists, by \Cref{lemma:arithmetic},  $q_j\in (2\pi i\overline{\QQ})^{k}$ such that 
\[
T_{R(p_j)}V = q_j + T_{R(p_j)}V_0,
\]
where $V$ is a translate of the $\QQ$-vector space $V_0.$
We choose some preimage $q_j'\in V$ such that $T_{R(p_j)}q_j' = q_j$. We do not claim that $q_j'$ is defined over a number field.
Then 
\[
\{(p_j, q_j'+v)\,|\, j \in J, v\in V_0\}\subseteq S\times V
\]
is Zariski dense in $S\times V$, since the Zariski closure contains $\{p_j\} \times V$ for all $j\in I$.

Therefore,
\[
\begin{gathered}
\overline{A(S\times V)}^{Zar} = \overline{A(\{(p_j, q_j'+v)\,|\, j\in J, v\in V_0\})}^{Zar} \\
=   \overline{  \{ (q_j + T_{R(p_j)}v, 2\pi i R(p_j))\,|\, j\in J, v\in V_0\} }^{Zar}
\end{gathered}
\]
The right-hand side is defined over $2\pi i\overline{\QQ}$.

Next, we address $(2)$.
Let $V$ be the minimal affine space associated with $S$.
The goal is to show that \[
\dim A(S\times V) = \dim T_R(V) = \dim S,
\] since then, $S$ is bi-algebraic, and we can apply (1).

For the rest of the proof, write $R$ and $\lambda$ for the constant value both take on $S$.
Let \[
l+1:= \rank \langle R_0,\ldots,R_n\rangle_{\QQ}\]
on $S$. After renumbering, assume that $R_0,\ldots, R_l$ is a $\QQ$-basis for the residues on $S$ and $l\geq 0.$
Using the $\QQ$-linear relations among the residues, one can rewrite
\[
\int_{x_0}^{x_k} \omega = F^{alg}_j + \sum_{j=0}^l R_j\log \hat{w}_{kj},
\]
where $\hat{w}_{kj}$ is a $\QQ$-linear combination of $w_{kj}$.
Set \[
\hat{v}:= (\hat{v}_{kj})_{k,j}, \hat{v}_{kj} := \log \hat{w}_{kj},\quad \hat{w} := (\hat{w}_{kj})_{k,j}.
\]
Let $\hat{W}$ be an algebraic torus with coordinates $\hat{w}$ and $\hat{V}$ the universal covering with coordinates $\hat{v}$.
We introduce \[
\hat{\Lambda}:= \hat{W}  \times \hat{V}\]
and \[
\begin{gathered}
\sigma:W\to \hat{W},\\
s\mapsto \hat{w}(s)
\end{gathered}
\]

Let $\hat{\Gamma}$ be the graph of the exponential function on $\hat{V}$ and
\[\hat{S}:=\sigma(S)\subseteq \hat{W}.
\]

The algebraic period map  can be rewritten  on $\hat{\Lambda}$ as 
\[
\begin{gathered}
\hat{A}:\hat{\Lambda}\rightarrow \calP,\\   (\hat{w},\hat{v})\mapsto( \hat{T}_{R}\hat{v}, 2\pi i R),
\end{gathered}
\]
where 
\[
\hat{T}_{R}\hat{v} = 2\pi i \left(\sum_{j=0}^{l} R_j \hat{v}_{1j},\ldots,\sum_{j=0}^l R_j\hat{v}_{nj}\right)
\]
In particular,
\[
\hat{A}(\sigma(s),\hat{v})  =A(s,v).
\]

By restricting $\hat{A}$ to the graph of the exponential function, it follows 
\[
\dim \hat{S} = \dim S.
\]

 We now claim that for a Zariski dense set of  arithmetic  points,  $\log \hat{w}_{kj}$ is a root of unity for all $k$ and $j\leq l$. To see this, note that by Baker's theorem $\log \hat{w}_{kj}\in 2\pi i \overline{\QQ}$. Otherwise, there would be an additional $\QQ$-linear relation among $R_0,\ldots, R_l$. Thus $\hat{w}_{kl}$ is a root of unity.
 Therefore, $\hat{S}\subseteq \hat{W}$ contains a Zariski dense set of torsion points.  It now follows from multiplicative Manin-Mumford, see for example \cite[Thm. 3.3]{Pila}, that the Zariski closure of $\hat{S}$ is a {\em torsion coset} $W'\subseteq \hat{W}$, i.e., a translate of a subtorus by a root of unity. Since $\hat{S}$ is constructible, it contains a Zariski dense open set of $W'$. 
 Choose a generic point $s\in S$ such that $\sigma$ is smooth in a neighborhood $U\subseteq S$ of $s$ and hence $\sigma(U)$ is open in $\hat{S}$. Let $S'$ be a component of $(\sigma(U)  \times \hat{V})\cap \hat{\Gamma}$. In particular, $S'$ is contained in $\hat{S}\times V'$, where $V'$ is the associated affine subspace to $W'$.
 Then, the image of the period map\[
 \phi(U) = \hat{A}(S')
 \]contains an (analytically) open subset of \[(T_R(V),2\pi iR)= (\hat{T}_R(V'),2\pi iR).
 \]
 Hence, $S$ is locally in period coordinate defined by $(T_R(V),2\pi iR)=A(S\times V)$ and therefore bi-algebraic.

To see (3), note that it follows from  $\fib(S)=0$ that $S$ is bi-algebraic by \Cref{lemma:fib}, (3) and thus we can apply (1).

\end{proof}
We end this section with an example of  a linear subvariety containing a Zariski dense set of arithmetic points.  Non-linear examples can be constructed  using subvarieties of the locus of exact differentials $\Exc$, but all of them have fiber rank zero. We do not know an example of a non-linear bi-algebraic variety with positive fiber rank and a Zariski dense set of arithmetic points.

\begin{exax}[A curve with a Zariski dense set of arithmetic points]
\label{ex:dense} If $\St$ has a single zero, then every algebraic variety $S$ defined over $\overline{\QQ}$ has a dense set of arithmetic points. In fact, every $S(\overline{\QQ})$-rational point is arithmetic. Using covering constructions, this example can be propagated into higher-dimensional strata. We will explain covering constructions in general in \Cref{sec:examples}.
Here, we will discuss a particular case. Consider the stratum $\St[(-1^4,1^2)]$ parametrized by 
\[
\omega = \lambda\dfrac{(z-1)(z-x_1)}{z(z-y_1)(z-y_2)}dz
\]
and consider the subvariety  $S = \{x_1 = -1, y_1= -y_2\}$ defined by differentials 
\[
\omega = \lambda\dfrac{(z-1)(z+1)}{z(z-\alpha)(z+\alpha)}dz
\]
This example arises by starting with $\St[(-1^3,1)]$ and pulling back differentials along the double cover $z\mapsto z^2$.
Note that here, we use a different normalization than in \Cref{sec-periodsetup}.
We claim that $r_1=r_2$, i.e., the residues at $\alpha$ and $-\alpha$ agree. This can be checked by hand, or it will follow from the description as covering construction. 
We embed  $\St[(-1^4,1^2)]$ into $(\CC^*)^4$ via 
\[
\iota:( x_1,y_1,y_2, \lambda)\mapsto (w_1 = x_1, w_2= \dfrac{x_1-y_1}{1-y_1}, w_3 = \dfrac{x_1-y_2}{1-y_2},\lambda).
\]
Local period coordinates are given by 
\[
\int_{-1}^1 \omega = R_0 \log w_1 + R_1 \log w_2 + R_2\log w_3, 2\pi i R_0, 2\pi i R_1, 2\pi i R_2. 
\]

The subvariety $S$ is embedded in $(\CC^*)^4$ as 
\[
w_1 =-1, w_2 w_3=1.
\]
It then follows from $\log w_1 \in 2\pi i \ZZ, R_2=R_3$, that a point in $S$ has local period coordinates given by

\[
 2\pi i(kR_0 , R_0,  R_1,R_1),
\]
where $k\in \ZZ$. In particular $(X,\omega)\in S$ is arithmetic if and only if $(X,\omega)$ is defined over $\overline{\QQ}$.
\end{exax}

\section{Examples of bi-algebraic varieties}\label{sec:examples}
We collect known examples of bi-algebraic varieties in strata of differentials.
For the rest of this section,   $\St$ can be a stratum of holomorphic or meromorphic differentials in any genus unless stated otherwise.

\subsection{Linear varieties}
A large class of known examples of bi-algebraic varieties arises from $\RR$-linear varieties in strata of holomorphic differentials, i.e., algebraic varieties locally in period coordinates defined by real subspaces. 
Filip proved in \cite{Filip} that, in holomorphic strata,  irreducible, $\RR$-linear bi-algebraic varieties are exactly the orbit closures for the $\GL(2,\RR)$-action on strata.
Examples of linear varieties in holomorphic strata with equations not defined over $\RR$ can be found in \cite{Moeller-Linear,KlinglerLerer,DeroinMatheus}.
In \cite{DeroinMatheus} an example of a bi-algebraic variety in a holomorphic  stratum was found, which is not linear.
\subsection{Residue varieties}
Let $\St$ be any stratum of meromorphic differentials.
The residue map $R:\St\to \CC^{|P|}$, where $|P|$ is the number of poles of $\St$ is algebraic. In particular for any algebraic variety $Z\subseteq \CC^{|P|}$ its preimage $R^{-1}(Z)$ is a bi-algebraic variety.
We call $R^{-1}(Z)$ a {\em residue variety}.
The image of the residue map is described by \cite{GendronTahar} and allows to find  in every stratum uncountably many residue varieties of any codimension.

\subsection{Covering constructions}
Let $\St[(\mu')]\to \Mgn[g',k']$ be any stratum and $S\subseteq \St[(\mu')]$ bi-algebraic. By pulling back differentials along branched coverings, one can produce new examples of bi-algebraic varieties.
Fix a ramification profile $\rho$ for branched coverings $f:X\to C, C\in \Mgn[g',k']$. The data of $\rho$ includes the number of branch points, the ramification multiplicities, and which ramification points lie in the same fiber of $f$.

Let \[H(\rho)=\{(X,C,f)\,|\,f:X\to C \text{ has ramification profile } \rho \}\]
be the Hurwitz space parametrizing covers with ramification profile $\rho$.
We will always assume that, on $C$, all branch points are marked on and,  on $X$, the preimages of all branch points are marked.

Consider the bundle of differentials
\[
\widetilde{S}(\rho)=\{(X,C,f,f^*\omega)\,|\, (X,C,f)\in H(\rho), (C,\omega) \in  S\}.
\]
Since the ramification profile is fixed, $(X,f^*\omega)$ is contained in a fixed stratum $\St.$
The image \[
S(\rho):=\pr(\widetilde{S}(\rho))\]  of the projection $\pr:\widetilde{S}(\rho)\to \St$  is a {\em covering construction}.

The source map $\pr_{H}:H(\rho)\to \Mgn, (X,C,f)\to X$ is finite, and hence the same is true for $\pr$. In particular, $S(\rho)$ is algebraic.

\begin{prop}
Let $S\subseteq \St[(\mu')]$ be a bi-algebraic variety and $\rho$ a ramification profile.
Then $S(\rho)\subseteq \St$ is bi-algebraic.
\end{prop}

\begin{proof}

We assumed that all branch points and ramification points  are marked and thus \[\dim S=\dim \widetilde{S}(\rho) =\dim S(\rho),\]
 where the last equality follows since, for a given curve  $X$, there are only finitely many morphisms to curves $C$. Here, we consider $C$ with marked points, and $f$ maps the marked points of $X$ the marked points of $C$. Thus, even if the genus is zero, there are only finitely many maps.

In case $S = \St$ it is  well known that $S(\rho)$ is bi-algebraic, locally defined by a union of  linear subspaces $f^*H^1(C-P,Z;\CC)$, where $f$ runs over all possible maps $f:X\to C$, see for example \cite[Thm. 4.4.1]{Zachhuber}.

Consider a general point $X_0$ of $S(\rho)$. We will construct countably many algebraic varieties in period coordinates that contain $S(\rho)$. Since $S(\rho)$ is smooth near $X_0$, it has to agree with one of the countably many components.

Let $(X_0,C_0,f_0,f_0^*\omega_0)\in\widetilde{S}(\rho)$ be a preimage of $X_0$. The covering $f_0$ is determined by a monodromy representation $\sigma:\pi_1(C_0,c_0)\to S_d$, where $d$ is the degree of $f_0$. 

On the universal covering $\widetilde{\St[(\mu')]}$ of $\St[(\mu')]$, we can trivialize the family of curves $C_0$ and thus identify the fundamental groups to construct a covering.
Using the monodromy representation, there is an induced map on universal coverings
\[
h(\sigma):\widetilde{\St[(\mu')]}\to \widetilde{\St},\quad
(C,\omega) \mapsto (X, f^*\omega)
\]
sending $C$ to the source of the covering $f:X\to C$ determined by the monodromy representation $\sigma$.
The following diagram commutes
\[
\begin{tikzcd}
\widetilde{\St[(\mu')]}\ar[r,"h(\sigma)"]\ar[d]& \widetilde{\St}\ar[d]\\
H^1(C_0-P(\omega_0),Z(\omega_0)) \ar[r, "p(\sigma)"]& H^1(X_0-P(f^*\omega_0),Z(f^*\omega_0)).
\end{tikzcd}
\]
Here, $p(\sigma)$ is the pullback in relative cohomology induced by $f:X_0\to C_0$ and is well-defined since we assumed that all preimages of branch points are marked.

Choose a local section $s:\St[(\mu')]\to \widetilde{\St[(\mu')]}$ defined near $(C_0,\omega_0)$. 
The image of the composition \[\St[(\mu')]\to \widetilde{\St[(\mu')]}\to \widetilde{\St} \to \St
\]
parametrizes a local irreducible component of $\St(\rho)$ near $X_0$.
By repeating the same procedure for the other preimages of $X_0$, this describes a whole neighborhood of $\St(\rho)$.

If $S$ is, locally near $(C_0,\omega_0)$,  defined by an algebraic variety $Z\subseteq H^1(C_0-P(\omega_0),Z(\omega_0))$  in period coordinates, then $S(\rho)$ is, locally near $(X_0,f^*\omega_0)$ contained in the union of all translates of $p(\sigma)(Z)$ under the action of $\pi_1(\St,X_0,f^*\omega)$ and the union over all possible monodromy representations $\sigma$. The fundamental group of a variety is countable, and thus
$S(\rho)$ is locally contained in a countable union of algebraic varieties of $\dim S$.
Since $S(\rho)$ is smooth near $X_0$, it follows that $S(\rho)$ agrees with one of the translates.
This shows $\rho(S)$ is bi-algebraic near a general point and thus bi-algebraic.

\end{proof}

The local  equations in period coordinates defining a covering construction in period coordinates can be computed using 
\[
\int_{\gamma} f^*\omega = \int_{f^*\gamma} \omega.
\]
In particular, if $\gamma$ and $\gamma'$ are two lifts of the same path the equation \[\int_{\gamma} f^*\omega = \int_{\gamma'} f^*\omega\] has to hold on $S(\rho)$.

Using covering constructions, we can produce many examples of bi-algebraic varieties whose algebraic models have peculiar behavior.
\begin{exax}[Zero fiber rank and Zariski dense arithmetic points]\label{ex:cov}
We now construct a bi-algebraic variety with fiber rank zero, which is not in the locus of exact differentials and contains a Zariski dense set of arithmetic points.
This is a continuation of the example \Cref{ex:dense}.
Start with the stratum $\St[(1,-1^3)]$. In particular $\St[(1,-1^3)]$ has fiber rank zero and every $\overline{\QQ}$-point is arithmetic. Let $S$ be the covering construction obtained from degree $2$ coverings, ramified at two  poles $y_1, y_2$ and unramified at the remaining pole $y_3$ and the zero $x_1$.
The resulting bi-algebraic variety $S$ is a $\CC^*$-invariant surface  contained in the $4$-dimensional stratum $\St[(1^2,-1^4)
]$.
One of the equations defining $S$ is given by $r_3=r_4$ where $r_3$
 and $r_4$ are the residue lying over the two preimages of $y_3$.
 The remaining equation comes from 
 \[
 \int_{\gamma} f^*\omega = \int_{f_*\gamma} \omega,
 \]
 where $\gamma$ is a relative period connecting the two zeros of $f^*\omega$.
 Since there are no relative periods on $(\PP^1-\{y_1,y_2,y_3\},\{z\})$ the period  $\int_{\gamma} f^*\omega$ is a sum of residues. 
 In particular, the map $S\mapsto R(S)$ is finite. If a bi-algebraic variety has positive fiber rank, then the fibers of the residue map $R_{|S}$ are positive dimensional as well.
 Furthermore, if $(C,\omega)$ in $\St[(1,-1^3)]$ is defined over $\overline{\QQ}$, so is the pullback $(X,f^*\omega)$ along the double covering $f:X\to C$.
And the periods of $f^*\omega$ satisfy
\[
\int_{\gamma} f^*\omega = \int_{f_*\gamma} \omega,
\]
Thus, if $\lambda\omega$ has periods in $\overline{\QQ}$, the same is true for $\lambda f^*\omega$. In particular, the pullback of an arithmetic point is arithmetic and hence $S$ contains a Zariski dense set of arithmetic points.
\end{exax}

\subsection{``Obvious" bi-algebraic varieties}
The intersection of bi-algebraic varieties is bi-algebraic, and in analogy to  M\"oller-Mullane \cite{MoellerMullane}, we call the intersection of any covering construction of a stratum with a residue variety an {\em ``obvious'' bi-algebraic variety}.
Below, we will construct examples of non-obvious bi-algebraic varieties in strata in genus zero with only simple poles.

\subsection{Log differentials}
Pulling back differentials is one way to produce bi-algebraic varieties. For branched coverings $f:X\to \PP^1$ one can also consider \[d \log f=\dfrac{df}{f}\]
instead.
Let $\rho$ be a ramification profile for branched coverings $f:X\to \PP^1$ and consider the corresponding Hurwitz space
 \[H(\rho)=\{(X,C,f)\,|\,f:X\to C \text{ has ramification profile } \rho \}.\]
 Here $C$ is $\PP^1$ with a collection of marked points.
 We mark all branch points of $f$ on $C$  as well as $0$ and $\infty$. On $X$, mark all preimages of marked points on $C$.
Over the Hurwitz space, we have the bundle of log differentials
\[
\widetilde{L}(\sigma)=\{(X,C,f,d\log f)\,|\, (X,C,f)\in H(\sigma)\}.
\]

The log differential $d \log f$ has a simple pole at every zero and every pole of $f$ as well as zeros at the remaining ramification points of $f$, not lying over $0$ or $\infty.$
If the local model of $f$ at a ramification  $x$ point is $z\mapsto z^n$, then \[\ord_{x} d\log f= n-1.
\]
In particular, for a fixed ramification profile $\rho$ the differential $d\log f$ is contained in a fixed stratum $\St$.

The image $L(\sigma)\subseteq \St$ under the forgetful map $\widetilde{L}(\sigma)\to \St$  is a {\em locus of log differentials}.
Since the forgetful map is finite and $\dfrac{df}{f}$ depends algebraically on $f$, the locus of log differentials is an algebraic variety.

\begin{prop}
Let $\rho$ be a ramification profile for branched coverings $f:X\to \PP^1$. Then $L(\sigma)$ is an affine-linear bi-algebraic variety.
\end{prop}

\begin{proof}
The idea is to find countably many affine-linear subspaces of dimension $\dim L(\rho)$ such that locally $L(\rho)$ is contained in the union. Since $L(\rho)$ is an analytic variety at a generic point, $L(\rho)$ coincides with one of the countably many affine spaces.
A differential $(X,\omega)\in \St$ is a log differential if and only if 
\[
\int_{\gamma} \omega \in 2\pi i\ZZ \text{ for all } \gamma \in H_1(X-P(\omega);\ZZ),
\]
since, in this case, 
\[
f(z) = e^{\int_{z_0}^z\omega},
\]
defines an an algebraic function $f:X\to \PP^1$ with $\omega= \dfrac{df}{f}$.

Therefore, $L(\rho)$ is locally  contained in a countable union of affine spaces of dimension  $2g(X) +|P(\omega)|-1$. 
Since $L(\rho)$ is an analytic variety, it has locally only finitely many components and is thus contained in finitely many affine spaces.

The dimension of the stratum is
\[
\dim \St = 2g(X) +|Z(f)|+ |P(f)|-1 + |\Ram(f)|-1,
\]
where $\Ram(f)$ is the set of ramification points of $f$, not including ramification over the marked points $0$ and $\infty.$

On the other hand
\[
\dim L(\rho)= |f(\Ram(f))|+2-3.
\] Here the plus two accounts for the marked points $0$ and $\infty$ on $\PP^1$.
Thus
\[
\codim L(\rho)= 2g(X) +|Z(f)|+ |P(f)|-1 + |\Ram(f)|-|f(\Ram(f))|.
\]

First, assume that $|f(\Ram(f))| = |\Ram(f)|$, i.e., over every branch point lies exactly one ramification point.
In that case \[
\codim_{\St} L(\rho) = 2g(X)+|P(\omega)|-1
\] and we are done.

For the general case, note that if two ramification points $p,p'$ of $f$ lie in the same fiber, then
\[
\int_p^{p'} d\log f \in 2\pi i \ZZ.
\]
Thus, on top of the $2g(X) +|Z(f)|+ |P(f)|-1$ linear equations, there is an additional set of $|\Ram(f)|-|f(\Ram(f))|$ linear equations among the relative periods satisfied on $L(\rho)$.
Since relative periods are part of a coordinate system, we have that in this case, $L(\rho)$ is contained in countably many affine spaces of dimension 
\[
 2g(X)+|Z(f)|+|P(f)|-1 + |\Ram(f)|-|f(\Ram(f))|
\]
and thus, the same argument as before goes through.
\end{proof}

\subsection{Exceptional linear varieties in strata with only simple poles}
Let $\St$ be a stratum in genus zero with only simple poles. We use the notation from 
\Cref{sec-periodsetup} .
The partial fraction decomposition and  also \cref{eq:Per} simplify to 
\[
\omega = \sum_{j=0}^m \dfrac{R_j(\lambda,x,y)}{z-y_j}.
\]  
and

\begin{equation}
\label{eq:periods_simple}
\int_{0}^{x_i}\omega =\sum_{j=0}^{m} R_j(\lambda,x,y)\log\dfrac{y_j-x_i}{y_j}
\end{equation}

In this section, we will construct a linear subvariety by locally describing linear equations in period coordinates.

Let $k\leq m$ and $M\in \CC^{(k+{n}) \times (m+n+1)}$ a matrix of the form
\[
M=\begin{pmatrix}
 A & B \\ 
 0 & C
\end{pmatrix}
\]
where $A\in \QQ^{k\times m}$ is a matrix in reduced row-echelon form,
$B\in \CC^{k\times{(n+1)}}$ and 
\begin{equation}
\label{eq:matrix}
C=
\left(
\begin{array}{ccccc}
 1 & 0 & \cdots & 0& q_1 \\
 0 & 1& \cdots&0 & \vdots \\
  \vdots & \vdots&\ddots  & \vdots &\vdots\\

 0 & 0&\cdots & 1 & q_n\\
\end{array}
\right)
\end{equation}
with $q=(q_1,\ldots, q_n)\in \QQ^n, q_i\neq 0\text{ for all } i.$
We will always assume that $M$ is in reduced row echelon form, i.e., only the last column of $B$ is potentially non-zero.

Let $U\subseteq \St$ be a period chart containing and 
\[
\phi_U(X,\omega) = (\int_{x_0}^{x_1} \omega, \ldots,\int_{x_0}^{x_m} \omega, 2\pi iR_0,\ldots,2\pi iR_n)
\]
the period map defined on $U$.
Define the analytic variety $S_M$ locally in period coordinates  by
\[
S_M := \{ (X,\omega)\in \St\,|\,  M \cdot \phi_U(X,\omega) =0\}.
\]

{\em A priori} $S_M$ is only well-defined locally on a period chart $U$. We have to check that the equations defining $S_M$ are invariant under the monodromy and thus glue to give a global algebraic variety.

\begin{prop}
Let $M$ be any matrix as in \cref{eq:matrix}.
The analytic variety $S_M\subseteq \St$ is a well-defined, bi-algebraic subvariety of expected codimension $k+n$.
Furthermore, $S_M$ is $\CC$-linear. If the matrix $B$ is defined over a field $k\subseteq \CC$, then $S_M$ is $k$-linear.
\end{prop}

\begin{proof}

First consider the residue variety $\mathfrak{R}_q\subseteq \St$  defined by
\[
\mathfrak{R}_q = \{R_j = q_jR_0 \text{ for } j =1,\ldots, n \},
\]
where $(q_1,\ldots,q_n)$ is the last column of the matrix $C.$ We set $q_0=1.$
Let $H$ be any linear equation in period coordinates
\[
H = \sum_{j=1}^m \int_0^{x_j} c_j \omega +\sum_{l=0}^{n} 2\pi id_lR_l,
\]
with $c_j \in \QQ, d_l\in \CC$.

{\em Claim:  The  zero locus of $H$ on $\mathfrak{R}_q$ is a closed algebraic subvariety.}
Assuming the claim, it follows that $S_M$ is a closed algebraic subvariety since it is an intersection of finitely many subvarieties $ \cap_{j=1}^k (\mathfrak{R} \cap \{H_j =0\}))$, where  $H_j$ is the linear equation defined by the $j$-th row of $M$. Since we defined $S_M$ by $k+n$ linearly independent equations in period coordinates, it also follows that $S_M$ is bi-algebraic of expected codimension $k+n$.

It remains to prove the claim.
Using \cref{eq:periods_simple}, we can rewrite $H$ as
\[
H = R_0\left[\sum_{i=1}^m\sum_{j=0}^n c_iq_j \ln\dfrac{y_j-x_i}{y_j} + \sum_{l=0}^{n} 2\pi id_lq_l\right].
\]

As long as $c_iq_j\in\QQ$ for all $i,j$, we can multiply with the greatest common divisor of all denominators, and thus the zero locus of $H$ is equal to
\begin{equation}\label{eq:linear}
0= b+\sum_{i=1}^m\sum_{j=0}^n a_{ij}\log\dfrac{ y_j-x_i}{y_j}\end{equation}
for $b\in\CC , a_{ij}\in\ZZ$.
Here, we used that in a stratum with  only simple poles, all residues are non-zero and hence $R_0$ can be factored out of the equation.
Exponentiating \cref{eq:linear} leads to the algebraic equation
\begin{equation}\label{eq:AlgEq}
Z_H:= c\prod_{k,j} y_j^{a_{kj}} - \prod_{k,j} \left(y_j -x_k\right)^{a_{kj}} = 0,
\end{equation}
for $c\in \CC^*$.
It follows that any component of  $\mathfrak{R}_q \cap \{ H =0\}$ is an irreducible component of $ \mathfrak{R}_q \cap \{ Z_H =0 \}.$

\end{proof}

It remains to produce examples where $S_M$ is non-empty.

\begin{prop}
Let $\St$ be a stratum in genus zero with only simple poles.
There exist infinitely many non-empty linear varieties $S_M\subseteq \St$.
Furthermore, there exist infinitely many  non-``obvious'' (meromorphic) Teichm\"uller curves.
\end{prop}

\begin{proof}

We first claim that
for infinitely many $q=(q_1,\ldots, q_n)\in (\QQ\setminus \{0\})^n$ the residue variety 
\[
\mathfrak{R}_q:=\{(X,\omega)\in\St \,|\, R_j = q_jR_0 \}
\]
is non-empty. This can be seen as follows. By \cite[Lemma 5.15.] {Tahar} every stratum contains a differential $(X,\omega)$ with all periods real. The flat picture is that of a collection of horizontal half-cylinders with identifications among the saddle connections on the core curves of the half-cylinders. By slightly deforming the saddle connections, one can arrange that all residues of $(X,\omega)$ are rational. Furthermore, a neighborhood of $(X,\omega)$ contains differentials lying in infinitely many residue varieties $\mathfrak{R}_q$. 
If $\mathfrak{R}_q$ is non-empty, then  $\codim_{\St} \mathfrak{R}_q=n$. 
Any differential $(X,\omega)\in \mathfrak{R}_q$ lies in some variety $S_M.$ In fact, one can choose the coefficients of $A$ to be arbitrary rational numbers and choose the last row of $B$ accordingly.

It remains to construct non-``obvious" Teichm\"uller curves.
The variety $S_M$ defines a curve exactly if 

\[
M=\left(
  \begin{array}{@{} c|c @{} | c | c}
    \begin{matrix}
      \quad\text{\fontsize{10mmm}{10mm}\selectfont$I$}\quad
    \end{matrix}
    &
    \begin{matrix}
      b_1\\\vdots\\b_{m-1} 
    \end{matrix}&
    \begin{matrix}
      \quad\text{\fontsize{10mmm}{10mm}\selectfont$0$}\quad
    \end{matrix}
    &
    \begin{matrix}
      p_1\\\vdots\\p_{m-1} 
    \end{matrix}
    \\
    \hline

\begin{matrix}
      \quad\text{\fontsize{10mmm}{10mm}\selectfont$0$}\quad
    \end{matrix}
    &
    \begin{matrix}
      0\\\vdots\\0 
    \end{matrix}&
    \begin{matrix}
      \quad\text{\fontsize{10mmm}{10mm}\selectfont$I$}\quad
    \end{matrix}
    &
    \begin{matrix}
      q_1\\\vdots\\q_{n} 
    \end{matrix}
  \end{array}
\right)
\]

If $S_M$ is ``obvious'' then the equations among residues can have arbitrary coefficients, but the equations among relative periods are defined over $\QQ$.
In other words the tangent space $T_{(X,\omega)} C$ of an obvious Teichm\"uller curve $C$ 
is a $2$-dimensional subspace of $H^1(X-P,Z;\CC)$ such that 
\[
T_{(X,\omega)} C \cap H^1(X,Z;\CC)
\]
is a subspace defined over the rationals.
For the variety $S_M$ the intersection  $T_{(X,\omega)}S_M \cap H^1(X,Z;\CC)$ is the subspace annihilated by the block matrix 
\[
M' = \begin{pmatrix}
A & B
\end{pmatrix}
\]
Since we assumed $M'$ is in reduced row-echelon form the field of definition for  $T_{(X,\omega)}S_M \cap H^1(X,Z;\CC)$ is the smallest subfield of $\CC$ containing all coefficients $M'$.
Thus, proceed as follows.
 Let $(X,\omega)\in\mathfrak{R}$ such that all periods are real. By perturbing $\omega$ slightly, we can assume that  all relative periods are in $\RR\setminus \QQ$. Choose $b_i=0$ for all $i$ and $p_i = -\int_{x_0}^{x_i} \omega$.
The resulting curve $S_M$ contains $(X,\omega)$ and  $T_{(X,\omega)}S_M \cap H^1(X,Z;\CC)$ is defined over the smallest field containing all relative periods.

\end{proof}
Using the same procedure, one can produce linear subvarieties defined over complex fields and in higher dimensions as well.

\end{document}